\documentclass[11pt,reqno,UTF8,twoside]{amsart}
\usepackage{mathrsfs}
\usepackage{amsfonts,amssymb,amsmath,amsthm}
\usepackage{cite}
\usepackage[colorlinks=true,citecolor=red]{hyperref}
\usepackage{titletoc}
\usepackage{geometry}
\usepackage{bm}
\usepackage{indentfirst}
\usepackage{graphicx}
\usepackage{float} 
\usepackage{booktabs}
\usepackage{longtable}
\usepackage{subfigure}
\usepackage{stmaryrd}
\usepackage{enumerate}
\usepackage{tikz}
\usepackage{ulem}
\usepackage{color,soul}
\usepackage{setspace}
\usepackage{stmaryrd}
\usepackage[pagewise]{lineno} 
\allowdisplaybreaks[2]
\usepackage{appendix}
\usepackage{cases}
\usepackage{todonotes}
\usepackage{enumitem}

\numberwithin{equation}{section}
\newtheorem{theorem}{Theorem}[section]
\newtheorem{lemma}[theorem]{Lemma}
\newtheorem{corollary}[theorem]{Corollary}
\newtheorem{proposition}[theorem]{Proposition}

\newtheorem{remark}[theorem]{Remark}
\newtheorem{definition}[theorem]{Definition}
\newtheorem{example}[theorem]{Example}

\newcommand{\N}{\mathbb{N}}
\newcommand{\R}{\mathbb{R}}
\newcommand{\Z}{\mathbb{Z}}
\newcommand{\E}{\mathbb{E}}

\newcommand{\leb}{\operatorname{Leb}}
\newcommand{\dist}{\operatorname{dist}}

\hypersetup{linkcolor=blue}

\makeatletter
\@namedef{subjclassname@2020}{\textup{2020} Mathematics Subject
Classification}
\makeatother

\begin{document}
     
   \title[ Strong Feller via moment method]
	{{\Large A {\MakeLowercase{note on the  strong} F{\MakeLowercase{eller property via the moment method}}}
    }}

    \author[Z. Liu, S. Xiang]{ {\small Z\MakeLowercase{iyu} L\MakeLowercase{iu}, S\MakeLowercase{hengquan} X\MakeLowercase{iang}}}
    
    \address[Ziyu Liu]{School of Mathematics and Physics, University of Science and Technology Beijing, 100083, Beijing, China.}
    \email{ziyu@ustb.edu.cn}

    \address[Shengquan  Xiang]{School of Mathematical Sciences, Peking University, 100871, Beijing, China.}
    \email{shengquan.xiang@math.pku.edu.cn}
    
\begin{abstract}
    This note studies the 1D stochastic heat equation  driven by a one-dimensional Brownian motion. We prove that the associated Markov process satisfies the strong Feller property under mild non-degeneracy conditions.  The approach combines Malliavin calculus with the moment method from PDE control theory.
\end{abstract}

	

\maketitle

\section{Introduction} \label{Sec 1}

 \subsection{Backgrounds and motivations}\label{Sec 1.1}

The strong Feller property plays a fundamental role in the study of long-time behavior of SPDEs.  When combined with irreducibility-type conditions, it often implies uniqueness of invariant measures and convergence to equilibrium; see, for example, the Doob theorem \cite{DPZ-96} and Harris-type theorems \cite{GM-05}. More recently, the strong Feller property has been established for singular SPDEs \cite{Hairer-18}, and has also been used in the study of Lyapunov exponents for stochastic dynamical systems \cite{BBPS-22c}.

For SPDEs, verifying the strong Feller property is in general a delicate issue, especially when the noise is degenerate. The classical monographs by Da Prato and Zabczyk \cite{DPZ-96,DPZ-92,DPZ-91} systematically study the strong Feller property and its connection with general SPDE models. In particular, for linear SPDEs, these results show that the strong Feller property is equivalent to the null-controllability of the associated deterministic system. In terms of smoothing properties, Goldys and Peszat \cite{GP-23,GP-26} recently derived differentiability of all orders for linear SPDEs by using the null-controllability and gradient estimates. Another widely used approach to the strong Feller property is based on Malliavin calculus: gradient estimates derived from the Bismut--Elworthy--Li formula (see e.g. \cite{EL-94,PZ-95,WZ-13,Zhang-13}) provide a convenient tool for this purpose.

\vspace{3mm}
In the present note, we study the strong Feller property for the 1D stochastic heat equation driven by a highly degenerate noise: {\it the equation is forced by a single Brownian motion}. The proof combines tools from PDE control theory with Malliavin calculus. On the one hand, Malliavin calculus relates the strong Feller property to null controllability of the associated deterministic system. On the other hand, in the 1D PDE setting, null controllability can be addressed using well-developed techniques such as the moment method.  While controllability issues are extensively studied for a wide range of deterministic systems, this perspective may also be useful in studying stochastic equations and noise structures beyond the reach of standard probabilistic techniques.

     \subsection{Main result}\label{Sec 1.2}
    Consider the stochastic heat equation, 
    \begin{equation}\label{H equation}
    \begin{cases}
         dy-\Delta ydt=f(x)dW_t,\quad x\in (0,1),\;t>0,\\
         y|_{x=0,1}=0,\\
        y(0,\cdot)=y_0(\cdot).
    \end{cases}
    \end{equation}
    Here $W$ denotes a one-dimensional standard Brownian motion. We use $\|\cdot\|$ and $\langle\cdot,\cdot\rangle$ to denote the usual  $L^2$-norm and inner product. Let $\{e_n\}_{n\in\N^+}$ be eigenfunctions of $-\Delta$ with Dirichlet boundary condition on $(0,1)$, with corresponding eigenvalues $\lambda_n$, i.e., 
    \begin{equation*}
        e_n(x)=\sqrt{2}\sin(\pi n x),\quad \lambda_n=\pi^2n^2,\quad n\in\N^+.
    \end{equation*}
    
    The spatial structure of the noise is described by the function 
    \begin{equation*}
        f\in H^{-s}(0,1):=\left\{f=\sum_{n\in\N^+} f_ne_n:\sum_{n\in\N^+}\lambda_n^{-s}|f_n|^2<\infty\right\}
    \end{equation*}
    with $0\leq s<1$.  Under these settings, by \cite[Section 5]{DPZ-96}, there exists a unique mild solution $y\in C([0,\infty);L^2(0,1))$.

    Recall that the strong Feller property (see \cite[Section 4.1]{DPZ-96}) is formulated as follows.
    \begin{definition}
    Let $\{\xi_t\}_{t\geq0}$ be a Markov process on a Polish space $Z$. We say that the Markov semigroup $\{P_t\}_{t\geq 0}$ associated with $\{\xi_t\}_{t\geq0}$ is strong Feller at time $T>0$ if for any bounded measurable function $\varphi\colon Z\rightarrow \R$, $\xi\mapsto P_T\varphi(\xi)$ is continuous on $Z$.      
    \end{definition}

     Our main result is  the following.
\begin{theorem}\label{thm H}
   For the equation \eqref{H equation}, assume there exist constants $C,\alpha>0$ and $s\in[0,1)$ such that $f\in H^{-s}(0,1)$ and
    \begin{equation}\label{eq f}
        |\langle f,e_n\rangle|\geq C e^{-\alpha \pi^2n^2}\quad \forall\, n\in\N^+.
    \end{equation}
    Then the Markov semigroup $\{P_t\}_{t\geq 0}$ on $L^2(0,1)$ associated with equation \eqref{H equation} is strong Feller at any time $T>\alpha$.
\end{theorem}

Let us mention that Theorem \ref{thm H} gives a strong Feller result for a rather degenerate noise. The degeneracy is twofold. From the probabilistic viewpoint, the equation is forced by only one Brownian motion, but the regularization is obtained for the infinite-dimensional Markov semigroup. This is also related to the asymptotic strong Feller theory for SPDEs with finite-dimensional degenerate noise; for instance, \cite{HM-06} established the asymptotic strong Feller property for the 2D stochastic Navier--Stokes equations driven by four-dimensional Brownian noise.

From the spatial viewpoint, the noise acts through a single profile $f$, which may be as singular as a point input $f=\delta_{p}$ for suitable $p\in(0,1)$; see Example \ref{ex1} below. This spatially singular forcing is in the same direction as the study of spatially localized noise acting on an open set, as considered in \cite{Shi-15, Shi-21, LWXZZ-24,LXZ-26,CXZZ-25}.

\begin{remark}
    It is worth noting that condition~\eqref{eq f} is rather mild, which admits a wide class of admissible noise inputs.  On the other hand, the result is also  sharp in the sense that if the noise input $f$ is degenerate in at least one Fourier mode, i.e., if
     \begin{equation*}
         \langle f,e_m\rangle=0\quad \text{for some }m\in\N^+,
     \end{equation*}
     then the associated Markov semigroup cannot satisfy the strong Feller property; see {\rm Remark \ref{rem1}} for more details.   
\end{remark}

      \begin{example}\label{ex1}
         Let $f=\delta_p$ be the Dirac delta distribution at a point $p\in(0,1)\setminus\mathbb Q$. Then
     \begin{equation*}
         |\langle \delta_p,e_n\rangle|=\sqrt{2}|\sin(\pi np)|,\quad n\in\N^+,
     \end{equation*}
     and $\delta_p\in H^{-s}(0,1)$ for any $s\in(1/2,1)$. Moreover, for every $\alpha>0$, there exists a Borel set $\Gamma_\alpha\subset(0,1)\setminus\mathbb Q$ with full Lebesgue measure such that, for every $p\in\Gamma_\alpha$, there exists a constant $C_{\alpha,p}>0$ satisfying
    \begin{equation*}
        |\langle \delta_p,e_n\rangle|\geq C_{\alpha,p} e^{-\alpha \pi^2n^2}\quad \forall\, n\in\N^+.
    \end{equation*}
     Indeed, for $q\in\R$,
        \begin{equation*}
            2\dist(q,\Z)\leq |\sin(\pi q)|\leq \pi\dist(q,\Z).
        \end{equation*}
        Consider the sequence of sets
        \begin{equation*}
            A_n:=\{p\in[0,1]:\dist(np,\Z)<e^{-\alpha \pi^2n^2}\}.
        \end{equation*}
        Since
        \begin{equation*}
            \leb(A_n)\leq 4e^{-\alpha\pi^2n^2},
        \end{equation*}
        we have $\sum_{n\geq1}\leb(A_n)<\infty$. By the Borel--Cantelli lemma, for a.e. $p\in[0,1]$ there exists $N(p)$ such that
        \begin{equation*}
            \dist(np,\Z)\geq e^{-\alpha\pi^2n^2},\quad n\geq N(p).
        \end{equation*}
        For irrational $p$, the remaining finitely many coefficients $|\sin(\pi np)|$, $1\leq n<N(p)$, are nonzero, which ensures the desired estimate.
     \end{example}

  While the present note focuses on the basic example of the 1D stochastic heat equation, the same approach extends more broadly. Related smoothing results of all orders for transition semigroups of linear SPDEs were obtained in \cite{GP-26}, via $n$-fold It\^o integral, gradient estimates and null-controllability. The proof of Theorem~\ref{thm H} combines Malliavin calculus with the moment method from PDE control theory.

    \begin{itemize}[leftmargin=2em]
        \item[\tiny$\bullet$] Malliavin calculus is a standard tool for establishing smoothing properties of Markov semigroups associated with SPDEs, typically through gradient estimates for the transition semigroup. In particular, the strong Feller property can often be obtained via the Bismut–Elworthy–Li formula; see, e.g. \cite{EL-94,Zhang-13}.  Malliavin-type arguments can also lead to the weaker notion of asymptotic strong Feller, introduced by Hairer and Mattingly \cite{HM-06}. See also \cite{Hairer-11,Nualart-06,LXZ-26} for further applications of Malliavin calculus, and \cite{GLLL-24,Szarek-06} for other smoothing properties of Markov semigroups.
    
        \item[\tiny$\bullet$]The moment method is a classical approach in control theory for studying controllability of linear evolution equations. It is based on spectral decompositions and reduces controllability to an infinite-dimensional moment problem; see, e.g. Fattorini--Russell~\cite{FR-71}. Also see  Beauchard\cite{Beauchard-05}, where the moment method is combined with Nash--Moser scheme. For 1D parabolic equations, this method provides explicit and constructive controllability results under mild assumptions, and has been widely used in the study of null and exact controllability; see, e.g. \cite{Russell-78,TW-09,Coron-07,GHXZ-22}.

    \end{itemize}

\section{Probabilistic setup and Malliavin calculus}\label{Sec 2}

In this section, we recall the Malliavin calculus approach to the strong Feller property for general SPDE models, which essentially relies on the Bismut--Elworthy--Li formula.

\vspace{3mm}
Let $(\mathcal H,\langle\cdot,\cdot\rangle_{\mathcal H})$ and $(U,\langle\cdot,\cdot\rangle_U)$ be two separable Hilbert spaces. We consider the abstract linear stochastic evolution equation
\begin{equation}\label{abs equation}
\begin{cases}
     dX(t)=AX(t)dt+BdW_t,\quad t>0,\\
     X(0)=x,
\end{cases}
\end{equation}
where $A:D(A)\subset \mathcal H\to \mathcal H$ generates an analytic semigroup $\{S(t)\}_{t\geq0}$ on $\mathcal H$. We assume that this semigroup is exponentially stable, namely, there exist constants $M\geq1$ and $c>0$ such that
\begin{equation*}
    \|S(t)\|_{\mathcal L(\mathcal H)}\leq Me^{-ct},\quad t\geq0.
\end{equation*}
Here $W$ is a cylindrical Wiener process on $U$.

\vspace{3mm}
Since $\{S(t)\}_{t\geq0}$ is exponentially stable, the generator $A$ is invertible; see, e.g., \cite{EN-00}. Following the interpolation--extrapolation framework of \cite[Section 6]{Amann-88}, we set $\mathcal H_0=\mathcal H$ and define $\mathcal H_{-1}$ as the completion of $\mathcal H$ with respect to the norm
\begin{equation*}
    \|x\|_{\mathcal H_{-1}}:=\|A^{-1}x\|_{\mathcal H},\quad x\in \mathcal H.
\end{equation*}
For $\eta\in(0,1)$, we set
\begin{equation}\label{eq H-eta}
    \mathcal H_{-\eta}:=[\mathcal H_{-1},\mathcal H]_{1-\eta},
\end{equation}
where $[\cdot,\cdot]_\theta$ denotes the complex interpolation functor. We endow $\mathcal H_{-\eta}$ with the corresponding interpolation norm given in \cite[Section 5]{Amann-88}.

\vspace{3mm}
By \cite[Theorem 6.1]{Amann-88}, $\{S(t)\}_{t\geq0}$ extends consistently to an analytic semigroup $\{S_{-\eta}(t)\}_{t\geq0}$ on $\mathcal H_{-\eta}$, and this extension satisfies
\begin{equation*}
    S_{-\eta}(t)x=S(t)x,\quad x\in \mathcal H,\;t\geq0.
\end{equation*}
Moreover, applying \cite[Theorem 6.1]{Amann-88}, for every $T>0$, there exists $C_T>0$ such that
\begin{equation}\label{eq extrap smoothing}
    \|S_{-\eta}(t)\|_{\mathcal L(\mathcal H_{-\eta};\mathcal H)}\leq C_Tt^{-\eta},\qquad 0<t\leq T.
\end{equation}
Consequently, if $B\in\mathcal L(U;\mathcal H_{-\eta})$, then $S_{-\eta}(t)B\in\mathcal L(U;\mathcal H)$ for every $t>0$.

\vspace{3mm}
We also assume that there exist $\eta\in[0,1)$ and $\beta\in(0,1/2)$ such that\footnote{This formulation allows $B$ to be {\it unbounded} as an operator from $U$ to $\mathcal H$, since it is only required to take values in the extrapolation space $\mathcal H_{-\eta}$.}
\begin{equation*}
    B\in\mathcal L(U;\mathcal H_{-\eta})
\end{equation*}
and, for every $T>0$,
\begin{equation}\label{HS condition}
    \int_0^T t^{-2\beta}\|S_{-\eta}(t)B\|_{\mathcal L_2(U;\mathcal H)}^2dt<\infty.
\end{equation}
Here $\mathcal L_2(U;\mathcal H)$ denotes the space of Hilbert--Schmidt operators from $U$ to $\mathcal H$, endowed with the Hilbert--Schmidt norm.

\vspace{3mm}
Under these assumptions,  by Lemma \ref{lem sto conv}, for every $x\in \mathcal H$ and $T>0$, equation \eqref{abs equation} admits a unique mild solution
\begin{equation*}
    X\in L^2(\Omega;C([0,T];\mathcal H)),
\end{equation*}
given by
\begin{equation}\label{mild abs equation}
    X(t)
    =
    S(t)x+\int_0^t S_{-\eta}(t-s)BdW_s,
    \qquad t\in[0,T].
\end{equation}

\vspace{3mm}
In particular, the strong Feller property is closely related to the null controllability of its linearization.  Specifically, consider the following controlled equation
  \begin{equation}\label{control equation}
    \begin{cases}
        dY=AYdt+Bh(t)dt,\quad t\in(0,T),\\
        Y_0=y,
    \end{cases}
    \end{equation}
    where $h\in L^2(0,T;U)$ denotes a deterministic control term. 

    \begin{definition}
        We say that $\{Y_t\}_{t\geq 0}$ is null-controllable at time $T>0$ if there exists  a constant $C>0$ such that for any $y\in \mathcal H$, there exists $h\in L^2(0,T;U)$ satisfying
        \begin{equation*}
            Y_T=0,\quad \|h\|_{L^2(0,T;U)}\leq C\|y\|_{\mathcal H} .
        \end{equation*}
    \end{definition}
    
\begin{proposition}\label{prop SF}
    Assume that system \eqref{control equation} is null-controllable at time $T>0$. Then there exists a constant $C>0$ such that for any $\varphi\colon \mathcal H\rightarrow \R$ with $\|\varphi\|_{\infty}$ and $\|\nabla\varphi\|_{\infty}$ finite, the Markov semigroup $\{P_t\}_{t\geq 0}$ on $\mathcal H$ associated with $\{X_t\}_{t\geq 0}$ satisfies that
    \begin{equation*}
        \|\nabla P_T\varphi(x)\|_{\mathcal H} \leq C\|\varphi\|_{\infty}\quad 
        \forall\, x\in \mathcal H.
    \end{equation*}
\end{proposition}

\begin{remark}
    A related result can be found in {\rm \cite[Theorem 7.2.1]{DPZ-96}} for linear SPDEs with bounded noise coefficient $B\in\mathcal{L}(U;\mathcal H)$. There, the strong Feller property of the associated Ornstein--Uhlenbeck semigroup is characterized by a null-controllability condition, through the structure of Gaussian measures and the Cameron--Martin theorem.

    Here we use Malliavin calculus and the Bismut--Elworthy--Li formula, following a standard route close to {\rm \cite{EL-94}} and its later variants. The proof is included for completeness and to emphasize that $\mathcal{D}X_t$ solves a deterministic control equation along Cameron--Martin directions, while null controllability provides a control compensating the linearized dynamics. In this sense, the Bismut--Elworthy--Li formula used here can be viewed as a duality relation between Malliavin differentiation and deterministic control trajectories.
\end{remark}

\begin{proof}[Proof of Proposition \ref{prop SF}]
    To indicate the initial condition and the random input, we also  denote $X_t$ by $X(t,x,W)$.  The Malliavin derivative $\mathcal{D}$  associated with $\{X_t\}_{t\geq 0}$ is then defined as, for each $T>0$ and $h\in L^2(0,T;U)$,
    \begin{equation*}
        \langle \mathcal{D}X_t,h\rangle_{L^2(0,T;U)}=\lim\limits_{\varepsilon\rightarrow 0}\frac{X(t,x,W+\varepsilon \int_0^\cdot h(s)ds)-X(t,x,W)}{\varepsilon},\quad t>0,
    \end{equation*}
    where the limit is taken in $L^2(\Omega;\mathcal H)$. Then for $V_t= \langle \mathcal{D}X_t,h\rangle_{L^2(0,T;U)}$, it follows that $V$ satisfies equation\footnote{Here $V$ is deterministic, since it is the linearized equation associated with a linear SPDE. In the nonlinear case, the corresponding formulation contains stochastic terms and is more involved; see, e.g., \cite{LXZ-26}.}
    \begin{equation*}
    \begin{cases}
        dV=AVdt+Bh(t)dt,\quad t\in(0,T),\\
        V_0=0,
    \end{cases}
    \end{equation*}
    Thus we compute that
    \begin{equation}\label{eq M-deri}
        \langle \mathcal{D}X_T,h\rangle_{L^2(0,T;U)}=\int_0^TS_{-\eta}(T-t)Bh(t)dt.
    \end{equation}

    \vspace{3mm}

    Now consider the mild solution of \eqref{control equation}, which has the form of 
    \begin{equation*}
        Y_t=S(t)y+\int_0^tS_{-\eta}(t-s)Bh(s)ds,\quad y\in \mathcal H,\; t>0.
    \end{equation*}
    Under the assumption of Proposition \ref{prop SF}, for any $y\in \mathcal H$, there exists $h\in L^2(0,T;U)$ such that
    \begin{equation*}
        Y_T=0,
    \end{equation*}
    implying that
    \begin{equation*}
        S(T)y=-\int_0^TS_{-\eta}(T-t)Bh(t)dt.
    \end{equation*}

    \vspace{3mm}
    Therefore, we then derive the Bismut--Elworthy--Li formula in the present setting. For any $y\in \mathcal H$, we have
    \begin{align*}
        \langle \nabla P_T\varphi(x),y\rangle_{\mathcal H}        &=\E\left(\nabla\varphi(X_T)S(T)y\right)\\
        &=-\E\left(\nabla\varphi(X_T)\int_0^TS_{-\eta}(T-t)Bh(t)dt\right)\\
        &=-\E\left(\nabla\varphi(X_T)\langle \mathcal{D}X_T,h\rangle_{L^2(0,T;U)}\right)\\
        &=-\E\langle \mathcal{D}\varphi(X_T),h\rangle\\
        &=-\E \left(\varphi(X_T)\int_0^T\langle h(t),dW_t\rangle_U\right),
    \end{align*}
    where the last equality follows from the integration by parts formula, see e.g. \cite{Nualart-06}.  Noting that $h\in L^2(0,T;U)$ is deterministic, the stochastic integral $\int_0^T\langle h(t),dW_t\rangle_U$ is thus a usual Itô integral.

    Consequently, we conclude that
    \begin{align*}
        |\langle \nabla P_T\varphi(x),y\rangle_{\mathcal H}|&\leq \left|\E \left(\varphi(X_T)\int_0^T\langle h(t),dW_t\rangle_U\right)\right|\leq \|\varphi\|_{\infty}\E\left|\int_0^T\langle h(t),dW_t\rangle_U\right|\\
        &\leq \|\varphi\|_{\infty}\left(\E\left|\int_0^T\langle h(t),dW_t\rangle_U\right|^2\right)^{\frac{1}{2}}=\|\varphi\|_{\infty}\|h\|_{L^2(0,T;U)}\\
        &\leq C\|\varphi\|_{\infty}\|y\|_{\mathcal H} . 
    \end{align*}
   This completes the proof of Proposition \ref{prop SF}.    
\end{proof}

 Taking into account \cite[Lemma 7.1.5]{DPZ-96},  Proposition \ref{prop SF} therefore implies the strong Feller property.

\begin{corollary}\label{coro SF}
    Assume that system \eqref{control equation} is null-controllable  at time $T>0$. Then the Markov semigroup $\{P_t\}_{t\geq 0}$ on $\mathcal{H}$ associated with $\{X_t\}_{t\geq 0}$ is strong Feller at time $T$. 
\end{corollary}

\section{Moment method for the null-controllability}\label{Sec 3}
In this section, we establish the null-controllability for the heat equation via the moment method, which is applied in the next section to derive the strong Feller property. 

\vspace{3mm}

To begin with, let us first consider the controlled heat equation, which reads  
 \begin{equation}\label{controlled H equation}
    \begin{cases}
         \partial_tz-\Delta z=f(x)h(t),\quad x\in (0,1),\;t\in(0,T),\\
         z|_{x=0,1}=0,\\
        z(0,\cdot)=z_0(\cdot).
    \end{cases}
    \end{equation}
    Here $h\in L^2(0,T)$ denotes the control term. 
    
    \begin{proposition}\label{prop H control} Assume that there exist constants $C,\alpha>0$ and $s\in[0,1)$ such that $f\in H^{-s}(0,1)$ and   \begin{equation*}
        |\langle f,e_n\rangle|\geq C e^{-\alpha \pi^2n^2}\quad \forall\, n\in\N^+.
    \end{equation*} Then system \eqref{controlled H equation} is null-controllable on $L^2(0,1)$ at any time $T>\alpha$.
    \end{proposition}

    \begin{remark}
    The proof below is based on the classical moment method and is standard; see, e.g., {\rm \cite{FR-71}}. We include the details for the reader's convenience and also to make explicit how the lower bound on the Fourier coefficients of $f$ enters the construction of the control.
    \end{remark}

    \begin{proof}[Proof of Proposition \ref{prop H control}]
    Recall that $e_n(x)=\sqrt{2}\sin(\pi n x)$ and $\lambda_n=\pi^2n^2$ denote the eigenfunction and eigenvalue of $-\Delta$ with Dirichlet boundary condition on $(0,1)$, respectively. Let us decompose both $y$ and $f$ on this basis by
    \begin{align*}
        &z(t,x)=\sum_{n\in\N^+}z_n(t)e_n(x),\quad z_n(t)=\langle z,e_n\rangle,\\
        &f(x)=\sum_{n\in\N^+}f_ne_n(x),\quad f_n=\langle f,e_n\rangle.
    \end{align*}
    With this notation, equation \eqref{controlled H equation} is equivalent to the ODE system
    \begin{equation}\label{eq H ode}
        z_n'(t)+\lambda_nz_n(t)=f_nh(t),\quad z_n(0)=z_{n,0}=\langle z_0,e_n\rangle,\quad n\in\N^+.
    \end{equation}
    Note that the solution of equation \eqref{eq H ode}  is given by
    \begin{equation*}
        z_n(t)=e^{-\lambda_nt}z_{n,0}+f_n\int_0^te^{-\lambda_n(t-s)}h(s)ds,\quad n\in\N^+,\;t>0.
    \end{equation*}
    Therefore, the null-controllability of system \eqref{controlled H equation} at time $T>0$ is reduced to find $h\in L^2(0,T)$ such that
    \begin{equation*}
        e^{-\lambda_nT}z_{n,0}+f_n\int_0^Te^{-\lambda_n(T-t)}h(t)dt=0,\quad \forall\, n\in\N^+.
    \end{equation*}
    Under our assumptions, $f_n\neq 0$ for each $n\in\N^+$. The above relation can be rewritten as
    \begin{equation}\label{eq H m-problem}
         \int_0^Te^{-\lambda_nt}h(T-t)dt=-\frac{e^{-\lambda_nT}}{f_n}z_{n,0},\quad \forall\, n\in\N^+.
    \end{equation}

    \vspace{3mm}
    In what follows, we shall determine a sequence $\{\theta_n\}_{n\in\N^+}$  of functions biorthogonal to the set $\{e^{-\lambda_nt}\}_{n\in\N^+}$ in $L^2(0,T)$, i.e.,
    \begin{equation}\label{eq H basis}
        \int_0^Te^{-\lambda_nt}\theta_m(t)dt=\begin{cases}
            1,\quad n=m,\\
            0,\quad n\neq m.
        \end{cases}
    \end{equation}
    In particular, using this expression, the moment problem \eqref{eq H m-problem} has a formal solution
    \begin{equation}\label{eq H h-def}
        h(t)=-\sum_{n\in\N^+}\frac{e^{-\lambda_nT}}{f_n}z_{n,0}\theta_n(T-t),\quad t\in(0,T).
    \end{equation}
    We then in the next step show that the series equation \eqref{eq H h-def} is convergent in $L^2(0,T)$,  thereby completing the proof of Proposition \ref{prop H control}.

    \vspace{3mm}
    Set 
    \begin{equation*}
        E_{n,T}:=\overline{\text{span}}\{e^{-\lambda_kt}\in L^2(0,T):k\in\N^+,k\neq n\},\quad n\in\N^+.
    \end{equation*}
    Noticing that $\sum_{n\in\N^+}\lambda_n^{-1}<\infty$, by the Müntz theorem, see e.g. \cite{Schwartz-59},  this leads that $e^{-\lambda_nt}\notin E_{n,T}$.   Thus there exists a unique function $r_n\in E_{n,T}$ such that lies closest to the function $e^{-\lambda_n t}$ in the $L^2(0,T)$-norm. We now define $\theta_n$ by
    \begin{equation*}
        \theta_n(t):=\frac{e^{-\lambda_nt}-r_n(t)}{d_{n,T}^2},\quad n\in\N^+,
    \end{equation*}
    where
    \begin{equation*}
        d_{n,T}:=\left(\int_0^T(e^{-\lambda_nt}-r_n(t))^2dt\right)^{\frac{1}{2}}.
    \end{equation*}
    By construction, relation \eqref{eq H basis} is satisfied.

    \vspace{3mm}

    It remains to verify the convergence of the series in equation \eqref{eq H h-def}. To begin with, note that, by construction
    \begin{equation*}
        \|\theta_n\|_{L^2(0,T)}=d_{n,T}^{-1}.
    \end{equation*}
    Additionally, the distance between $e^{-\lambda_nt}$  and the closed subspace 
     \begin{equation*}
        E_{n,\infty}:=\overline{\text{span}}\{e^{-\lambda_kt}\in L^2(\R^+):k\in\N^+,k\neq n\}
    \end{equation*}
    is given by
    \begin{equation*}
        d_{n,\infty}^2=\frac{1}{2\lambda_n}\prod_{k\neq n}\left(\frac{\lambda_k-\lambda_n}{\lambda_k+\lambda_n}\right)^2,
    \end{equation*}
    see e.g. \cite{MZ-11}. Moreover, by \cite{Schwartz-59}, the injection mapping from $E_{n,\infty}$ into $E_{n,T}$, carrying each function in $E_{n,\infty}$ into its restriction to $(0,T)$ admits a bounded inverse. As a result, for each $T>0$, there exists a constant $C_T>0$ such that
    \begin{equation*}
        d_{n,T}^2\geq C_T^{-2}d_{n,\infty}^2\quad\forall\,n\in\N^+.
    \end{equation*}

    Summarizing these relations, we derive that
    \begin{align*}
        \|\theta_n\|_{L^2(0,T)}&=d_{n,T}^{-1}\leq C_Td_{n,\infty}^{-1}\\
        &\leq C_T\sqrt{2\lambda_n}\left|\prod_{k\neq n}\frac{\lambda_k+\lambda_n}{\lambda_k-\lambda_n}\right|\leq C_T\sqrt{2}\pi n\left|\frac{\prod_{k\neq n}(1+\frac{n^2}{k^2})}{\prod_{k\neq n}(1-\frac{n^2}{k^2})}\right|\\
        &=C_T\sqrt{2}\pi n\frac{\sinh(\pi n)}{\pi n}\leq Ce^{\pi n}.
    \end{align*}

    Consequently, invoking this inequality and our assumptions, we conclude that
    \begin{align*}
        \|h\|_{L^2(0,T)}&\leq \sum_{n\in\N^+}\frac{e^{-\lambda_nT}}{|f_n|}|z_{n,0}|\|\theta_n\|_{L^2(0,T)}\\
        &\leq C\sum_{n\in\N^+}\frac{e^{-T\pi^2n^2+\pi n}}{|f_n|}|z_{n,0}|\\
        &\leq C\|z_0\|\left(\sum_{n\in\N^+}e^{-2(T-\alpha)\pi^2n^2+2\pi n}\right)^{\frac{1}{2}}\\
        &\leq C\|z_0\|,
    \end{align*}
    provided that $T>\alpha$. Since the series defining $h$ converges in $L^2(0,T)$, testing it against $e^{-\lambda_n(T-\cdot)}\in L^2(0,T)$ and using the biorthogonality relation \eqref{eq H basis} shows that \eqref{eq H m-problem} holds for every $n\in\N^+$. The proof of Proposition \ref{prop H control} is thus completed.
    
    \end{proof}

\begin{remark}\label{rem1}
    The condition \eqref{eq f} shows that sufficiently non-degenerate forcing in every Fourier mode leads to null controllability, and hence to the strong Feller property. In this sense, the condition is close to being sharp: if the forcing degenerates completely in one mode, then both properties fail.

    \vspace{2mm}

    Indeed, if $f_m=0$ for some $m\in\N^+$, then the $m$-th component of \eqref{eq H ode} becomes
    \begin{equation*}
        z_m'(t)+\lambda_mz_m(t)=0,
        \quad z_m(0)=z_{m,0}.
    \end{equation*}
    Hence
    \begin{equation*}
        z_m(T)=e^{-\lambda_mT}z_{m,0}.
    \end{equation*}
    Taking $z_{m,0}\neq 0$, we obtain $z_m(T)\neq 0$ for every $T>0$. Thus the system cannot be null-controllable on $L^2(0,1)$.

    \vspace{2mm}

    The same obstruction also appears for the strong Feller property of the corresponding linear stochastic equation driven by the noise direction $f$. If $f_m=0$, then the $m$-th Fourier mode of the solution is deterministic:
    \begin{equation*}
        X_m(T)=e^{-\lambda_mT}x_m.
    \end{equation*}
    Let $\psi\colon\R\rightarrow\R$ be a bounded measurable function which is not continuous, and set
    \begin{equation*}
        \varphi(x)=\psi(\langle x,e_m\rangle).
    \end{equation*}
    Then
    \begin{equation*}
        P_T\varphi(x)=\psi(e^{-\lambda_mT}\langle x,e_m\rangle).
    \end{equation*}
    Since $e^{-\lambda_mT}\neq 0$, $P_T\varphi$ is not continuous on $L^2(0,1)$. Therefore $P_T$ is not strong Feller.
\end{remark}

    \section{Strong Feller property for the stochastic heat equation}\label{Sec 4}

    With the preparations from Sections \ref{Sec 2}-\ref{Sec 3}, we are now able to establish strong Feller property for the stochastic heat equation \eqref{H equation}, i.e. Theorem \ref{thm H}.

    \vspace{3mm}

    \begin{proof}[Proof of Theorem \ref{thm H}]        
    To align with the abstract setting in Section \ref{Sec 2}, let us set
\begin{equation*}
    \mathcal H=L^2(0,1),\quad U=\R,\quad A=\Delta \text{ with Dirichlet boundary condition on }(0,1),
\end{equation*}
\begin{equation*}
    B\colon\R\rightarrow H^{-s}(0,1),\quad r\mapsto rf(\cdot),
\end{equation*}
and
\begin{equation*}
    \mathcal H_{-\eta}=H^{-s}(0,1)\footnote{Here $\mathcal H_{-\eta}$ is the extrapolation space defined by \eqref{eq H-eta}. Under the Dirichlet spectral decomposition, it is characterized by $\sum_{n\in\N^+}\lambda_n^{-2\eta}|g_n|^2<\infty$, and therefore coincides with the spectral negative Sobolev space $H^{-s}(0,1)$ when $\eta=s/2$.}
    =
    \left\{
    f=\sum_{n\in\N^+} f_ne_n:
    \sum_{n\in\N^+}\lambda_n^{-s}|f_n|^2<\infty
    \right\}
    \quad \text{with}\quad \eta=\frac{s}{2}.
\end{equation*}
Then $B\in\mathcal L(\R;H^{-s}(0,1))$. Additionally, for any $\beta\in(0,\frac{1-s}{2})$ and $T>0$,
    \begin{equation*}
    \int_0^Tt^{-2\beta}\|e^{t\Delta }f\|^2dt\leq C\|f\|^2_{H^{-s}}\int_0^Tt^{-s-2\beta}dt<\infty.
    \end{equation*}

    Then by Proposition \ref{prop SF} and Corollary \ref{coro SF}, to complete the proof Theorem \ref{thm H}, it suffices to verify the null-controllability  of the associated deterministic control system
    \begin{equation*}
    \begin{cases}
         \partial_tz-\Delta z=f(x)h(t),\quad x\in (0,1),\;t\in(0,T),\\
         z|_{x=0,1}=0,\\
        z(0,\cdot)=z_0(\cdot),
    \end{cases}
    \end{equation*}
    where $h\in L^2(0,T)$ denotes the control.

    This null controllability property is precisely guaranteed by Proposition~\ref{prop H control}. Therefore, the strong Feller property stated in Theorem~\ref{thm H} follows, and the proof is complete.

    \end{proof}

    \appendix

\section{Well-posedness of the stochastic convolution}\label{app A}

In this appendix, we recall a standard consequence of the factorization method, following \cite[Theorem 5.2.6]{DPZ-96}, to justify the $\mathcal H$-valued continuity of the stochastic convolution used in Section~\ref{Sec 1.2}.

\begin{lemma}\label{lem sto conv}
    Let $T>0$. Assume that $B\in\mathcal L(U;\mathcal H_{-\eta})$ for some $\eta\in[0,1)$ and that there exists $\beta\in(0,1/2)$ such that
    \begin{equation}\label{A HS condition}
        \int_0^T t^{-2\beta}\|S_{-\eta}(t)B\|_{\mathcal L_2(U;\mathcal H)}^2dt<\infty.
    \end{equation}
    Then the stochastic convolution
    \begin{equation*}
        W_A(t):=\int_0^tS_{-\eta}(t-s)BdW_s,\quad t\in[0,T],
    \end{equation*}
    is well defined and admits an $\mathcal H$-valued continuous modification. Moreover,
    \begin{equation*}
        W_A\in L^2(\Omega;C([0,T];\mathcal H)).
    \end{equation*}
\end{lemma}

\begin{proof}
    By the Itô isometry, for every $t\in[0,T]$,
    \begin{align*}
        \E\|W_A(t)\|_{\mathcal H}^2&=\int_0^t\|S_{-\eta}(t-s)B\|_{\mathcal L_2(U;\mathcal H)}^2ds=\int_0^t\|S_{-\eta}(r)B\|_{\mathcal L_2(U;\mathcal H)}^2dr<\infty,
    \end{align*}
    where the last inequality follows from \eqref{A HS condition}. Hence $W_A(t)$ is well defined as an $\mathcal H$-valued random variable.

    We next prove the continuity. Choose $\gamma\in(0,\beta)$ and define
    \begin{equation*}
        Y_\gamma(t):=\int_0^t(t-s)^{-\gamma}S_{-\eta}(t-s)BdW_s.
    \end{equation*}
    Again by the Itô isometry and \eqref{A HS condition},
    \begin{align*}
        \sup_{t\in[0,T]}\E\|Y_\gamma(t)\|_{\mathcal H} ^2\leq T^{2(\beta-\gamma)}\int_0^T r^{-2\beta}\|S_{-\eta}(r)B\|_{\mathcal L_2(U;\mathcal H)}^2dr<\infty.
    \end{align*}
    Since $Y_\gamma(t)$ is Gaussian, for every $p\geq2$, 
    \begin{equation}\label{A Ygamma}
        Y_\gamma\in L^p(\Omega;L^p(0,T;\mathcal H)).
    \end{equation}

    Take $p>1/\gamma$. By the factorization formula,
    \begin{equation}\label{A factorization}
        W_A(t)=\frac{\sin(\pi\gamma)}{\pi}\int_0^t(t-r)^{\gamma-1}S(t-r)Y_\gamma(r)dr,
        \quad t\in[0,T].
    \end{equation}
    Indeed, this follows from stochastic Fubini's theorem, the consistency of the semigroups, and the identity
    \begin{equation*}
        \int_s^t(t-r)^{\gamma-1}(r-s)^{-\gamma}dr=\frac{\pi}{\sin(\pi\gamma)}.
    \end{equation*}

    By \cite[Proposition A.1.1(ii)]{DPZ-96}, since $p>1/\gamma$, the deterministic operator
\begin{equation*}
    \mathcal R_\gamma z(t):=\int_0^t(t-r)^{\gamma-1}S(t-r)z(r)dr
\end{equation*}
maps $L^p(0,T;\mathcal H)$ continuously into $C^\delta([0,T];\mathcal H)$ for every $\delta\in(0,\gamma-\frac1p)$. In particular, it maps $L^p(0,T;\mathcal H)$ continuously into $C([0,T];\mathcal H)$.

    Combining this mapping property with \eqref{A Ygamma} and \eqref{A factorization}, we obtain
    \begin{equation*}
        W_A\in L^p(\Omega;C([0,T];\mathcal H)).
    \end{equation*}
    In particular, choosing $p\geq2$ gives
    \begin{equation*}
        W_A\in L^2(\Omega;C([0,T];\mathcal H)).
    \end{equation*}
    This completes the proof.
\end{proof}
    
    \normalem
    \bibliographystyle{plain}
    \bibliography{References}

\begin{thebibliography}{10}

\bibitem{Amann-88}
H.~Amann.
\newblock Parabolic evolution equations in interpolation and extrapolation spaces.
\newblock {\em J. Funct. Anal.}, 78(2):233--270, 1988.

\bibitem{Beauchard-05}
K.~Beauchard.
\newblock Local controllability of a {1-D Schrödinger} equation.
\newblock {\em J. Math. Pures Appl.}, 9(84):851--956, 2005.

\bibitem{BBPS-22c}
J.~Bedrossian, A.~Blumenthal, and S.~Punshon-Smith.
\newblock Lagrangian chaos and scalar advection in stochastic fluid mechanics.
\newblock {\em J. Eur. Math. Soc. (JEMS)}, 24(6):1893--1990, 2022.

\bibitem{CXZZ-25}
Y.~Chen, S.~Xiang, Z.~Zhang, and J.-C. Zhao.
\newblock Exponential mixing for the randomly forced {NLS} equation.
\newblock {\em arXiv:2506.10318}, 2025.

\bibitem{Coron-07}
J.-M. Coron.
\newblock {\em Control and nonlinearity}.
\newblock American Mathematical Society, Providence, RI, 2007.

\bibitem{DPZ-91}
G.~Da~Prato and J.~Zabczyk.
\newblock Smoothing properties of transition semigroups in {Hilbert} spaces.
\newblock {\em Stochastics}, 35:63--77, 1991.

\bibitem{DPZ-92}
G.~Da~Prato and J.~Zabczyk.
\newblock {\em Stochastic equations in infinite dimensions}.
\newblock Cambridge University Press, Cambridge, 1992.

\bibitem{DPZ-96}
G.~Da~Prato and J.~Zabczyk.
\newblock {\em {E}rgodicity for infinite dimensional systems}.
\newblock Cambridge University Press, Cambridge, 1996.

\bibitem{EL-94}
K.~D. Elworthy and X.-M. Li.
\newblock Formulae for the derivatives of heat semigroups.
\newblock {\em J. Funct. Anal.}, 125(1):252--286, 1994.

\bibitem{EN-00}
K.-J. Engel and R.~Nagel.
\newblock {\em One-Parameter Semigroups for Linear Evolution Equations}.
\newblock Springer, New York, 2000.

\bibitem{FR-71}
H.~O. Fattorini and D.~L. Russell.
\newblock Exact controllability theorems for linear parabolic equations in one space dimension.
\newblock {\em Arch. Rational Mech. Anal.}, 43:272--292, 1971.

\bibitem{GHXZ-22}
L.~Gagnon, A.~Hayat, S.~Xiang, and C.~Zhang.
\newblock Fredholm transformation on {L}aplacian and rapid stabilization for the heat equation.
\newblock {\em J. Funct. Anal.}, 283(12):Paper No. 109664, 67 pp, 2022.

\bibitem{GM-05}
B.~Goldys and B.~Maslowski.
\newblock Exponential ergodicity for stochastic {B}urgers and 2{D} {N}avier-{S}tokes equations.
\newblock {\em J. Funct. Anal.}, 226(1):230–255, 2005.

\bibitem{GP-23}
B.~Goldys and S.~Peszat.
\newblock Linear parabolic equation with {Dirichlet} white noise boundary conditions.
\newblock {\em J. Differential Equations}, 362:382--437, 2023.

\bibitem{GP-26}
B.~Goldys and S.~Peszat.
\newblock Differentiability of transition semigroup of generalized {Ornstein--Uhlenbeck} process: a probabilistic approach.
\newblock {\em Ann. Sc. Norm. Super. Pisa Cl. Sci.}, 2026.
\newblock 29 pp.

\bibitem{GLLL-24}
F.~Gong, Y.~Liu, Y.~Liu, and Z.~Liu.
\newblock Ergodicity for eventually continuous {Markov--Feller semigroups on Polish spaces}.
\newblock {\em Sci. China Math.}, 2026, early access.

\bibitem{Hairer-11}
M.~Hairer.
\newblock On {M}alliavin's proof of {Hörmander's} theorem.
\newblock {\em Bull. Sci. Math.}, 135(6-7):650--666, 2011.

\bibitem{Hairer-18}
M.~Hairer.
\newblock The strong {F}eller property for singular stochastic {PDE}s.
\newblock {\em Ann. Inst. Henri Poincaré Probab. Stat.}, 54(3):1314--1340, 2018.

\bibitem{HM-06}
M.~Hairer and J.~C. Mattingly.
\newblock Ergodicity of the 2{D} {N}avier-{S}tokes equations with degenerate stochastic forcing.
\newblock {\em Ann. of Math. (2)}, 164(3):993--1032, 2006.

\bibitem{LWXZZ-24}
Z.~Liu, D.~Wei, S.~Xiang, Z.~Zhang, and J.-C. Zhao.
\newblock Exponential mixing for random nonlinear wave equations: weak dissipation and localized control.
\newblock {\em arXiv:2407.15058}, 2024.

\bibitem{LXZ-26}
Z.~Liu, S.~Xiang, and J.-C. Zhao.
\newblock Exponential mixing for the stochastic {Allen--Cahn} equation with localized white noise.
\newblock {\em arXiv:2605.06009}, 2026.

\bibitem{MZ-11}
S.~Micu and E.~Zuazua.
\newblock Regularity issues for the null-controllability of the linear 1-d heat equation.
\newblock {\em Systems Control Lett.}, 60(6):406--413, 2011.

\bibitem{Nualart-06}
D.~Nualart.
\newblock {\em The {M}alliavin calculus and related topics}.
\newblock Springer-Verlag, Berlin, second edition, 2006.

\bibitem{PZ-95}
S.~Peszat and J.~Zabczyk.
\newblock Strong {Feller} property and irreducibility for diffusions on {Hilbert} spaces.
\newblock {\em Ann. Probab.}, 23:157--172, 1995.

\bibitem{Russell-78}
D.~L. Russell.
\newblock Controllability and stabilizability theory for linear partial differential equations: recent progress and open questions.
\newblock {\em SIAM Rev.}, 20(4):639--739, 1978.

\bibitem{Schwartz-59}
L.~Schwartz.
\newblock {\em Étude des sommes d'exponentielles. 2ième éd}.
\newblock Publications de l'Institut de Mathématique de l'Université de Strasbourg, V. Actualités Sci. Ind., Hermann, Paris, 1959.

\bibitem{Shi-15}
A.~Shirikyan.
\newblock Control and mixing for 2{D} {N}avier-{S}tokes equations with space-time localised noise.
\newblock {\em Ann. Sci. \'{E}c. Norm. Sup\'{e}r}, 48(2):253--280, 2015.

\bibitem{Shi-21}
A.~Shirikyan.
\newblock Controllability implies mixing {II}. {C}onvergence in the dual-{L}ipschitz metric.
\newblock {\em J. Eur. Math. Soc. (JEMS)}, 23(4):1381--1422, 2021.

\bibitem{Szarek-06}
T.~Szarek.
\newblock {F}eller processes on nonlocally compact spaces.
\newblock {\em Ann. Probab.}, 34(5):1849--1863, 2006.

\bibitem{TW-09}
M.~Tucsnak and G.~Weiss.
\newblock {\em Observation and control for operator semigroups}.
\newblock Birkhäuser Verlag, Basel, 2009.

\bibitem{WZ-13}
F.-Y. Wang and X.-C. Zhang.
\newblock Derivative formula and applications for degenerate diffusion semigroups.
\newblock {\em J. Math. Pures Appl.}, 99:726--740, 2013.

\bibitem{Zhang-13}
X.~Zhang.
\newblock Derivative formulas and gradient estimates for {SDEs} driven by $\alpha$-stable processes.
\newblock {\em Stochastic Process. Appl.}, 123(4):1213--1228, 2013.

\end{thebibliography}

\end{document}